\documentclass[11pt, reqno]{amsart}

\usepackage{latexsym}
\usepackage{amssymb}
\usepackage{mathrsfs}
\usepackage{amsmath}
\usepackage{amsfonts}     
\usepackage{enumerate}
\usepackage[latin1]{inputenc}
\usepackage{eurosym}
\usepackage{color}

\newcommand{\kom}[1]{}
\renewcommand{\kom}[1]{{\bf [#1]}}

 \def\1{\raisebox{2pt}{\rm{$\chi$}}}

\newtheorem{theorem}{Theorem}[section]
\newtheorem{corollary}[theorem]{Corollary}
\newtheorem{lemma}[theorem]{Lemma}

\theoremstyle{definition}
\newtheorem{definition}[theorem]{Definition}
\newtheorem{remark}[theorem]{Remark}
\newtheorem{example}[theorem]{Example}
\numberwithin{equation}{section}

\newcommand{\R}{{\mathbb R}}

\newcommand{\loc}{{\text{loc}}}

 \newcommand{\eps}{{\varepsilon}}
\newcommand{\ph}{{\varphi}}
 \def\1{\raisebox{2pt}{\rm{$\chi$}}}
 
\newcommand{\abs}[1]{\left|#1\right|}

\def\vint_#1{\mathchoice%
          {\mathop{\kern 0.2em\vrule width 0.6em height 0.69678ex depth -0.58065ex
                  \kern -0.8em \intop}\nolimits_{\kern -0.4em#1}}%
          {\mathop{\kern 0.1em\vrule width 0.5em height 0.69678ex depth -0.60387ex
                  \kern -0.6em \intop}\nolimits_{#1}}%
          {\mathop{\kern 0.1em\vrule width 0.5em height 0.69678ex
              depth -0.60387ex
                  \kern -0.6em \intop}\nolimits_{#1}}%
          {\mathop{\kern 0.1em\vrule width 0.5em height 0.69678ex depth -0.60387ex
                  \kern -0.6em \intop}\nolimits_{#1}}}
\def\vintslides_#1{\mathchoice%
          {\mathop{\kern 0.1em\vrule width 0.5em height 0.697ex depth -0.581ex
                  \kern -0.6em \intop}\nolimits_{\kern -0.4em#1}}%
          {\mathop{\kern 0.1em\vrule width 0.3em height 0.697ex depth -0.604ex
                  \kern -0.4em \intop}\nolimits_{#1}}%
          {\mathop{\kern 0.1em\vrule width 0.3em height 0.697ex depth -0.604ex
                  \kern -0.4em \intop}\nolimits_{#1}}%
          {\mathop{\kern 0.1em\vrule width 0.3em height 0.697ex depth -0.604ex
                  \kern -0.4em \intop}\nolimits_{#1}}}

\newcommand{\aveint}[2]{\mathchoice%
          {\mathop{\kern 0.2em\vrule width 0.6em height 0.69678ex depth -0.58065ex
                  \kern -0.8em \intop}\nolimits_{\kern -0.45em#1}^{#2}}%
          {\mathop{\kern 0.1em\vrule width 0.5em height 0.69678ex depth -0.60387ex
                  \kern -0.6em \intop}\nolimits_{#1}^{#2}}%
          {\mathop{\kern 0.1em\vrule width 0.5em height 0.69678ex depth -0.60387ex
                  \kern -0.6em \intop}\nolimits_{#1}^{#2}}%
          {\mathop{\kern 0.1em\vrule width 0.5em height 0.69678ex depth -0.60387ex
                  \kern -0.6em \intop}\nolimits_{#1}^{#2}}}

\newcommand{\vp}{\varphi}
\newcommand{\bd}{\partial}

\newcommand{\esssup}{\operatornamewithlimits{ess\, sup}}
\newcommand{\essinf}{\operatornamewithlimits{ess\,inf}}

\newcommand{\B}{\mathcal B}

\begin{document}

\title[]{Supercaloric functions for the porous medium equation}

\author{}
\author[J.\! Kinnunen]{Juha Kinnunen}  
\address[J. Kinnunen]{Department of Mathematics, Aalto University, P. O. Box 11100, FI-00076 Aalto University, Finland}
\email{juha.k.kinnunen@aalto.fi}

\author[P.\! Lehtel\"a]{Pekka Lehtel\"a}  
\address[P. Lehtel\"a]{Department of Mathematics, Aalto University, P. O. Box 11100, FI-00076 Aalto University, Finland}
\email{pekka.lehtela@aalto.fi}

\author[P.\! Lindqvist]{Peter Lindqvist}  
\address[P. Lindqvist]{Department of Mathematics, Norwegian University of Science and Technology, N-7491 Trondheim,
Norway}
\email{peter.lindqvist@ntnu.no}

\author[M.\! Parviainen]{Mikko Parviainen}   
\address[M. Parviainen]{University of Jyvaskyl\"a, Department of Mathematics and Statistics, P. O. Box 35, FI-40014 University of Jyvaskyl\"a, Finland}
\email{mikko.j.parvianen@jyu.fi}

\thanks{The research is supported by the Academy of Finland, the Emil Aaltonen Foundation and the Norwegian Research Council.}

\begin{abstract}
In the slow diffusion case unbounded supersolutions of the porous medium equation are of two totally different types,
depending on whether the pressure is locally integrable or not.
This criterion and its consequences are discussed.
\end{abstract}

\maketitle

\section{Introduction}
\label{sec:intro}

The porous medium equation 
\begin{equation}\label{eq:pme}
u_t-\Delta(u^m)=0,
\quad m>1,
\end{equation}
has a well developed theory for its solutions treated, for example, in the monographs 
\cite{dibenedetto12}, \cite{daskalopoulosk07}, \cite{vazquez07} and \cite{wu01}.
Here $u=u(x,t)$ is a non-negative function on $\Omega_T=\Omega\times(0,T)$, where $\Omega\subset\R^n$.
We are interested in supersolutions of \eqref{eq:pme} in the slow diffusion case $m>1$.
A supersolution should satisfy the inequality $u_t-\Delta(u^m)\ge0$, but this is a delicate issue.
For a function $u:\Omega_T\to[0,\infty]$, we consider the two definitions below:
\begin{itemize}
\item  (Weak supersolution) We say that $u$ is a weak supersolution, if $u^m\in L_{\loc}^2(0,T; W^{1,2}_{\loc}(\Omega))$ and
\begin{equation}\label{eq:pme_weak}
\iint_{\Omega_T} \left( -u\varphi_t+ \nabla (u^m) \cdot \nabla \varphi\right) \, dx \, dt\ge 0
\end{equation}
for every non-negative test function $\vp\in C^\infty_0(\Omega_T)$. 
\item ($m$-supercaloric function)
We say that $u$ is $m$-supercaloric, if 
  \begin{itemize}
  \item[(i)] $u$ is lower semicontinuous,
  \item[(ii)] $u$ is finite in a dense subset of $\Omega_T$ and
  \item[(iii)] $u$ obeys the comparison principle with respect to solutions in every subdomain $D_{t_1,t_2}=D\times(t_1,t_2)$, $D_{t_1,t_2}\Subset \Omega_T$:
if $h\in C(\overline{D_{t_1,t_2}})$ is a weak solution of \eqref{eq:pme} in $D_{t_1,t_2}$ and $u\ge h$ on the parabolic boundary $\bd_p D_{t_1,t_2}$, then $u\ge h$ in $D_{t_1,t_2}$.
\end{itemize}
\end{itemize}
In a similar manner, we may also consider solutions defined, for example, in $\Omega\times(-\infty,\infty)$ or in $\R^n\times\R$.
Since our results are local, we may restrict ourselves to space-time cylinders in $\R^{n+1}$. 
The case $m=1$ gives supercaloric functions for the heat equation. 

Several remarks are appropriate.
First, weak supersolutions obey the comparison principle, see \cite{daskalopoulosk07}, \cite{vazquez07} and \cite{wu01}, and by \cite{avelin15} they are lower semicontinuous, after a possible redefinition on a set of $(n+1)$-dimensional Lebesgue measure zero.
Thus every weak supersolution is an $m$-supercaloric function.
Second, by \cite{kinnunenl08} every bounded $m$-supercaloric function is a weak supersolution.
In particular, they belong to the natural Sobolev space.
Thus if we only consider bounded functions then the classes of $m$-supercaloric functions and weak supersolutions coincide.
The advantage of weak supersolutions is that they satisfy expedient Caccioppoli and Harnack estimates.
The class of non-negative $m$-supercaloric functions is even more flexible. 
For example, it is closed under increasing convergence, if the limit function is finite on a dense set.
It is closed under taking minimum of finitely many $m$-supercaloric functions.
Moreover, a non-negative  $m$-supercaloric function may be redefined to be zero until a given moment of time. 
These properties will be useful for us later.

The examples below show that $m$-supercaloric functions are by no means innocent. 
Let $q_1,q_2,\dots$ be the points with rational coordinates in $\R^n$. 
The stationary function
\[
u(x,t)
=\left(\sum_{i=1}^\infty\frac{a_i}{|x-q_i|^{n-2}}\right)^{\frac1m},
\quad n\ge3,
\]
is a weak supersolution in $\R^{n+1}$, provided that the coefficients $a_i>0$ are chosen properly.
Now $u(q_i,t)\equiv\infty$, $i=1,2,\dots$. Nevertheless, $u^m\in L^2_{\loc}(\R^n;W^{1,2}_{\loc}(\R^n))$.

In general, the class of $m$-supercaloric functions is wider: it contains important functions that fail to be weak supersolutions.
The two important examples for us are:
\begin{itemize}
\item The Barenblatt solution
\begin{equation}\label{barenblatt}
\B(x,t)=
\begin{cases}
\displaystyle t^{-\lambda}\left(C-\frac{\lambda(m-1)}{2mn}\frac{\abs{x}^2}{t^{\frac{2\lambda}n}}\right)_+^{\frac1{m-1}}, \quad t>0,\\
0,\quad t\le 0,  
\end{cases}
\end{equation}
where $m>1$, $\lambda= \frac n {n(m-1)+2}$ and  $C>0$.
This function is $m$-supercaloric in $\R^{n+1}$, but it is not a weak supersolution in any domain that contains the origin, since
\[
\int_{-1}^1\int_{|x|<r}|\nabla(\B(x,t)^m)|^2\,dx\,dt=\infty.
\]
However, it is a weak solution in $\R^{n+1}\setminus\{0\}$. 
Moreover, $\B\in L^q_{\loc}(\R^n\times\R)$ whenever $q<m+\frac2n$, the weak gradient exists and $\nabla(\B^m)\in L^q_{\loc}(\R^n\times\R)$ for every $q<1+ \frac 1{1+mn}$, see \cite{kinnunenl08}. 
Furthermore, 
\[
\B_t - \Delta (\B^m) = C\delta,
\]
where $\delta$ is Dirac's delta and $C=C(m,n)>0$.
\item The friendly giant
\begin{equation}\label{friendly giant}
V(x,t) =
\begin{cases}
\displaystyle\frac{U(x)}{(t-t_0)^{\frac 1 {m-1}}}, \quad t>t_0,\\
0,\quad t\le t_0,
\end{cases}
\end{equation}
where $x\in\Omega$, $\Omega\subset\R^n$ is a bounded domain. 
Here $U>0$ satisfies the auxiliary elliptic equation
\begin{equation}
\label{eq:elliptic-eq}
\Delta (U^m) + \frac 1 {m-1} U = 0
\end{equation}
with the zero boundary values in $\Omega$.
This function is $m$-supercaloric, but it is not a weak supersolution in $\Omega\times \R$.
However, $V$ is a weak solution in $\Omega\times (t_0,\infty)$.
A characteristic feature of the friendly giant is a total blow-up at a time slice
\[
\lim_{{\substack{(y,t) \to (x,t_0)\\t>t_0}}} V(y,t)=\infty 
\quad\text{for every}\quad x\in \Omega.
\]
See  \cite [p. 111--114]{vazquez07}. 
\end{itemize}

Notice that $\B^{m-1}\in L_{\loc}^1(\R^{n+1})$ while $V^{m-1}\notin L^1_{\loc}(\Omega_T)$.
This summability for the pressure is decisive. 
Unbounded $m$-supercaloric functions are divided into two mutually exclusive classes $\mathfrak{B}$ and $\mathfrak{M}$
depending on whether the pressure $\frac{u^{m-1}}{m-1}$ is locally integrable or not.
The following results were outlined in \cite{kinnunen16} and our aim now is to provide complete proofs.
Let $m>1$. 
For an $m$-supercaloric function $u:\Omega_T\to[0,\infty]$ the following conditions are equivalent:

\begin{itemize}
\item {\bf Class $\mathfrak B$}
  \begin{itemize}
  \item[(i)] $u\in L_{\loc}^{m-1}(\Omega_T)$,
  \item[(ii)] $u\in L_{\loc}^q(\Omega_T)$ whenever $q<m+\frac2n$,
\item[(iii)] the Sobolev gradient $\nabla (u^m)$ exists and belongs to $L_{\loc}^q (\Omega_T)$  whenever $q< 1+ \frac 1 {1+mn}$,
 \end{itemize}
The proof is given in Section \ref{sec:B}.
Notice the gap $[m-1,m+\frac2n)$. 
Furthermore, functions of class $\B$ satisfy a measure equation
\[
u_t-\Delta(u^m) = \mu,
\]
where $\mu$  is a non-negative a Radon measure on $\R^{n+1}$.

\item {\bf Class $\mathfrak M$}
\begin{itemize}
\item[(i)]  $u\not \in L_{\loc}^{m-1}(\Omega_T)$,
\item[(ii)]  there exists a time $t_0\in (0,T)$ such that 
\[
\lim_{{\substack{(y,t) \to (x,t_0)\\t>t_0}}} u(y,t)=\infty 
\quad\text{for every}\quad x\in \Omega.
\]
\end{itemize}

The proof is given in Section \ref{sec:M}.
Functions of class $\mathfrak M$ have very few, if any good properties. 
In particular, they do not induce a Radon measure.
 \end{itemize}

Finally we also study the infinity sets, where 
\[
\lim_{{\substack{(y,t) \to (x,t_0)\\t>t_0}}} u(y,t)=\infty.
\]
Throughout we assume that $u>0$, but the assumption $u\ge0$ would be more appropriate because of the moving boundary.
We are mainly interested in $m$-supercaloric functions on sets where they are unbounded.
\section{Preliminaries}

Let $\Omega$ be an open, bounded and connected subset of $\R^n$ and let $0<t_1<t_2<T$.
We denote $\Omega_T=\Omega\times(0,T)$ and $D_{t_1,t_2}=D\times(t_1,t_2)$, where $D\subset\Omega$ is open.
The parabolic boundary of a space-time cylinder $D_{t_1,t_2}$ is $\bd_p D_{t_1,t_2}=(\overline D\times\{t_1\})\cup(\partial D\times[t_1,t_2))$,
that is, it consists of the initial and lateral boundaries. 
$D_{t_1,t_2}\Subset\Omega_T$ denotes that $\overline{D_{t_1,t_2}}$ is a compact subset of $\Omega_T$.

We use $W^{1,p}(\Omega)$, $1\le p<\infty$, to denote the Sobolev space of functions $u\in L^p(\Omega)$, whose weak gradients also belong to $L^p(\Omega)$, with the norm 
\[
\|u\|_{W^{1,p}(\Omega)}=\|u\|_{L^p(\Omega)}+\|\nabla u\|_{L^p(\Omega)}.
\] 
The Sobolev space with zero boundary values $W^{1,p}_0(\Omega)$ is the completion of $C^\infty_0(\Omega)$ with respect to the norm of $W^{1,p}(\Omega)$. Moreover, $u\in W^{1,p}_{\loc}(\Omega)$ if $u\in W^{1,p}(D)$ for every $D\Subset\Omega$.

The parabolic Sobolev space $L^2(0,T; W^{1,2}(\Omega))$ consists of measurable functions $u:\Omega_T\to[-\infty,\infty]$ such that
$x\mapsto u(x,t)$ belongs to $W^{1,2}(\Omega)$ for almost every $t\in(0,T)$ and
\[
\iint_{\Omega_T}(|u|^2+|\nabla u|^2)\,dx\,dt<\infty.
\]
The definition for $L^2(0,T; W^{1,2}_0(\Omega))$ is similar apart from the requirement that 
$x\mapsto u(x,t)$ belongs to $W^{1,2}_0(\Omega)$ for almost every $t\in(0,T)$.
Moreover, we say that $u\in L_{\loc}^2(0,T; W^{1,2}_{\loc}(\Omega))$ 
if $u\in L^2(t_1,t_2; W^{1,2}(D))$ for every $D_{t_1,t_2}\Subset\Omega_T$.

\begin{definition}\label{weak super}
A non-negative function $u$ is a weak supersolution to
\eqref{eq:pme}, if $u^m\in L_{\loc}^2(0,T; W^{1,2}_{\loc}(\Omega))$ and it satisfies \eqref{eq:pme_weak}
for every non-negative test function $\vp\in C^\infty_0(\Omega_T)$. 
Similarly, $u$ is a weak subsolution, if \eqref{eq:pme_weak} holds with the inequality reversed.
Moreover, $u$ is a weak solution to \eqref{eq:pme}, if the integral in \eqref{eq:pme_weak} is zero for every test function $\vp\in C^\infty_0(\Omega_T)$ without the sign restriction.
\end{definition}

In order to obtain appropriate Caccioppoli type energy estimates it is convenient to impose the Sobolev space assumption on $u^{\frac{m+1}2}$ instead of $u^m$ in the definition above. 
According to the recent result in \cite{lehtela17}, this does not make any difference for locally bounded functions.

\begin{theorem}
Assume that $u$  is a locally bounded non-negative function.
\begin{itemize}
\item[(i)] If $u^{\frac{m+1}2}\in L_{\loc}^2(0,T; W^{1,2}_{\loc}(\Omega))$ satisfies \eqref{eq:pme_weak} 
for every non-negative test function $\vp\in C^\infty_0(\Omega_T)$, then  $u^m\in L_{\loc}^2(0,T; W^{1,2}_{\loc}(\Omega)$. 
\item[(ii)] If  $u^m\in L_{\loc}^2(0,T; W^{1,2}_{\loc}(\Omega)$ satisfies \eqref{eq:pme_weak} 
for every non-negative test function $\vp\in C^\infty_0(\Omega_T)$, then $u^{\frac{m+1}2}\in L_{\loc}^2(0,T; W^{1,2}_{\loc}(\Omega))$.
\end{itemize}
\end{theorem}

\begin{remark}
Under the assumption $u^{\frac{m+1}2}\in L_{\loc}^2(0,T; W^{1,2}_{\loc}(\Omega))$ the gradient in \eqref{eq:pme_weak} is interpreted as
\[
\nabla (u^m)=\frac{2m}{m+1}u^{\frac{m-1}2}\nabla (u^{\frac{m+1}2}).
\]
\end{remark}

We point out that it we may restrict ourselves to bounded weak supersolutions, since we consider the truncations 
\[
u_j=\min\{u,j\}, \quad j=1,2,\dots.
\] 
In addition, we assume that $u$ is positive throughout, since we are interested in sets where functions are large.

We shall investigate several aspects related to unbounded $m$-supercaloric functions.
Harnack type estimates with intrinsic geometry play a fundamental role in our study.
Positive weak solutions to the porous medium equation satisfy the following intrinsic Harnack inequality, see \cite[Theorem 3]{dibenedetto88}, \cite{dibenedetto12}, \cite{daskalopoulosk07} and \cite{wu01}.

\begin{lemma}[Harnack]\label{harnack sol}
Assume that $u$ is a positive weak solution to \eqref{eq:pme} in $\Omega_T$. 
Then there exist constants $C_1$ and $C_2$, depending on $n$ and $m$, such that
\[
u(x_0,t_0)\le C_1 \inf_{x\in B(x_0,r)} u(x, t_0+\theta),
\] 
where 
\[
\theta = \frac{C_2 \rho^2}{u(x_0,t_0)^{m-1}}
\]
is such that $B(x_0,2r)\times (t_0-2\theta, t_0+2\theta)\subset\Omega_T$. 
\end{lemma}

For locally bounded positive weak supersolutions, we have the corresponding weak Harnack estimate, see \cite[Theorem 17.1, p. 133]{dibenedetto12} and \cite{lehtela16}.

\begin{lemma}[Weak Harnack]\label{Harnack}
Assume that $u$ is a locally bounded positive weak supersolution to \eqref{eq:pme} in $\Omega_T$ and let $B(x_0,8r) \times (0,T)\subset\Omega_T$. 
Then there exist constants $C_1$ and $C_2$, depending only on $m$ and $n$, such that for almost every $t_0\in (0,T)$, we have
\[
\vint_{B(x_0,r)} u(x,t_0)\, dx \le \left( \frac {C_1r^2}{T-t_0} \right)^{\frac1{m-1}} + C_2 \essinf_Q u,
\]
where $Q=B(x_0,4r)\times (t_0+\frac\theta2,t_0+\theta)$ with
\[
\theta= \min \left\{ T-t_0, C_1r^2 \left( \vint_{B(x_0,r)} u(x,t_0) \, dx \right)^{-(m-1)} \right\}.
\]
\end{lemma}

We shall discuss several results related to local integrability of $m$-supercaloric functions.
Caccioppoli type energy estimates allow us to derive estimates for local integrability of the gradient in terms of the 
local integrability of a supersolution.

\begin{lemma}[Caccioppoli]\label{caccioppoli}
Assume that $u$ is a locally bounded positive weak supersolution to \eqref{eq:pme}  in $\Omega_T$ and let $\zeta\in C_0^\infty(\Omega_T)$ be a cut-off function such that $0\le \zeta \le 1$. Then there exist numerical constants $C_1$ and $C_2$ such that
  \begin{align*}
 &\iint_{\Omega_T} m u^{m-\eps-2} \zeta^2 |\nabla u|^2 \, dx \, dt + \frac 1 {\eps | 1-\eps|} \esssup_{t\in(0,T)} \int_\Omega u^{1-\eps} \zeta^2 \, dx \\
& \le \frac {C_1 m}{\eps^2} \iint_{\Omega_T} u^{m-\eps} |\nabla \zeta|^2 \, dx \, dt + \frac{C_2}{\eps |1-\eps|} 
\iint_{\Omega_T} u^{1-\eps} \zeta |\zeta_t| \, dx \, dt
  \end{align*}
for every $\eps>0$, $\eps\ne 1$.
\end{lemma}

\begin{proof}
The Caccioppoli estimate follows by choosing the test function $\vp=u^{-\eps}\zeta^2$ combined with technical smoothing and dampening arguments. 
For a detailed proof, we refer to \cite[Lemma 2.4]{lehtela16}. 
\end{proof}

In the case $\eps=1$, the Caccioppoli estimate takes the following logarithmic form.

\begin{lemma}\label{caccioppoli2}
Assume that $u$ is a locally bounded positive weak supersolution to \eqref{eq:pme} in $\Omega_T$ and let $\zeta\in C_0^\infty(\Omega_T)$ be a cut-off function 
such that $0\le \zeta \le 1$. Then there exist numerical constants $C_1$ and $C_2$ such that
  \begin{align*}
 &  \iint_{\Omega_T}  m u^{m-3} \zeta^2 |\nabla u|^2 \, dx \, dt +  \esssup_{t\in(0,T)} \left |\int_\Omega \zeta^2 \log u \, dx \right | \\
& \le  C_1 m \iint_{\Omega_T}  u^{m-1} |\nabla \zeta|^2 \, dx \, dt + C_2 \iint_{\Omega_T} \zeta |\zeta_t||\log u| \, dx \, dt. 
  \end{align*}  
\end{lemma}

\begin{proof}
Let $0\le t_1 < t_2\le T$. Formally, we apply  $\vp=u^{-1}\zeta^2$ as a test function in the inequality
\[
\int_{t_1}^{t_2} \int_\Omega \left ( \nabla (u^m) \cdot \nabla \vp + \vp \frac {\partial u}{\partial t} \right) \, dx \, dt \ge 0.
\]
We observe $u^{-1} \frac {\partial u}{\partial t} = \frac {\partial \log(u)}{\partial t}$ and integrate by parts to get
\begin{equation}\label{log estimate1}
\begin{split}
&\int_{t_1}^{t_2} \int_\Omega \Big(- m u^{m-3} | \nabla u|^2 \zeta^2 + 2mu^{m-2} \zeta \nabla u \cdot \nabla \zeta - 2 \zeta \zeta_t \log u \Big)  \, dx \, dt \\
&+ \int_\Omega \zeta(x,t_2)^2 \log u(x,t_2) \, dx -  \int_\Omega \zeta(x,t_1)^2 \log u(x,t_1) \, dx \ge 0.
\end{split}
\end{equation}
Young's inequality gives
\begin{equation}\label{young}
\begin{split}
|2mu^{m-2} \zeta \nabla u \cdot \nabla \zeta| 
&\le   2mu^{m-2} \zeta |\nabla u||\nabla \zeta| \\
&\le \frac m 2 ( \zeta^2 u^{m-3}|\nabla u |^2 + 4  u^{m-1}|\nabla \zeta |^2). 
\end{split}
\end{equation}
By combining  \eqref{log estimate1} and \eqref{young} we arrive at
\begin{align*}
  &\int_{t_1}^{t_2} \int_\Omega \frac m2 u^{m-3} |\nabla u|^2 \zeta^2 \, dx \,dt + \int_\Omega \zeta(x,t_1)^2 \log u(x,t_1) \, dx \\
&\qquad- \int_\Omega \zeta(x,t_2)^2 \log u(x,t_2) \, dx \\
&\le 2 \int_{t_1}^{t_2} \int_\Omega u^{m-1} |\nabla \zeta|^2 \, dx \, dt +2\int_{t_1}^{t_2} \int_\Omega \zeta |\zeta_t| |\log u| \, dx \, dt.
\end{align*}
In the previous estimate we may first take supremum over $t_1$ and then let $t_2\rightarrow T$ or first take supremum over $t_2$ and then let $t_1\rightarrow
0$. It follows that
 \begin{align*}
  &\int_{0}^{T} \int_\Omega m u^{m-3} |\nabla u|^2 \zeta^2 \, dx \,dt + \esssup_{t\in (0,T)}\left|\int_\Omega \zeta^2 \log u \, dx \right|\\
&\le 4 \int_{0}^{T} \int_\Omega u^{m-1} |\nabla \zeta|^2 \, dx \, dt +4\int_{0}^{T} \int_\Omega \zeta |\zeta_t| |\log u| \, dx \, dt. 
\end{align*}
\end{proof} 

Finally, we recall a parabolic Sobolev's inequality, which is a tool to conclude local integrability estimates for a function in terms
of its gradient. 

\begin{lemma}[Sobolev]\label{parabolic sobolev}
Assume that $u \in L^p(0,T;W^{1,p}(\Omega))$ and $\zeta\in C_0^\infty(\Omega_T)$. 
Then there exists constant $C$, depending only on $n$, such that
\[
\iint_{\Omega_T}  |\zeta u|^q \, dx \, dt 
\le C^q \iint_{\Omega_T}  |\nabla(\zeta u)|^p \, dx \, dt \left ( \esssup_{t\in(0,T)} \int_\Omega |\zeta u|^r \, dx \right)^{\frac pn},
\]
where $r>0$ can be chosen as we please and $q=p+\frac{pr}n$.
\end{lemma}

\begin{proof}
See \cite[p. 7--8]{dibenedetto93}.
\end{proof}
\section{Characterizations for class $\mathfrak{B}$}\label{sec:B}

In this section we consider $m$-supercaloric functions that have a similar behaviour as the Barenblatt solution. 
We say that a positive $m$-supercaloric function $u$ belongs to class $\mathfrak{B}$, if $u\in L_{\loc}^{m-1}(\Omega_T)$.  
In other words, the pressure $\frac{u^{m-1}}{m-1}$ is locally integrable.
The following theorem gives several characterizations for these functions.

\begin{theorem}\label{class B}
Assume that $u$ is a positive $m$-supercaloric function in $\Omega_T$.
Then the following properties are equivalent:
  \begin{itemize}
  \item[(i)] $u\in L_{\loc}^q(\Omega_T)$ for some $q>m-1$,
  \item[(ii)] $u\in L_{\loc}^{m-1}(\Omega_T)$,
\item[(iii)] $\nabla (u^m)$ exists and $\nabla (u^m) \in L_{\loc}^q (\Omega_T)$  whenever $q< 1+ \frac 1 {1+mn}$,
\item[(iv)] 
$\displaystyle\esssup_{t\in(\delta,T-\delta)} \int_D u(x,t)\, dx < \infty$
whenever $D\Subset\Omega$ and $\delta\in(0,\frac T2)$. 
 \end{itemize}
\end{theorem}

\begin{proof}
First we prove the theorem in the case, when (iv) is replaced with the following slightly weaker condition:
\begin{itemize}
\item[(iv')] there exists $\alpha\in(0,1)$ such that
\begin{equation}\label{eq:ualpha}
\esssup_{t\in(\delta,T-\delta)} \int_D u(x,t)^\alpha \, dx < \infty
\end{equation}
whenever $D\Subset\Omega$ and $\delta\in(0,\frac T2)$. 
 \end{itemize}
We show that (i) $\Longleftrightarrow$ (iii)  and (i) $\Longleftrightarrow$ (iv').
The remaining equivalences are treated in Remark \ref{class B concluded}.

First we show that (i) implies (iii). This follows from \cite[Theorem 1.4]{kinnunenl08}. 
However, there is a missing assumption in   \cite[Theorem 1.4]{kinnunenl08} and the existence of the Poisson modification was taken for granted in \cite[Section 5]{kinnunenl08}. 
In order to complete the proof, we show that the Poisson modification exists under the assumption $u\in L_{\loc}^q(\Omega_T)$ with $q>m-1$. 

Let $Q\Subset\Omega$ be a cube and let $Q'\Subset Q$ be a subcube of $Q$. 
Fix $t_1\in(0,T)$.
We redefine $u$ by setting $u(x,t)=0$ when $(x,t)\in Q\times(0,t_1)$.
Observe that the redefined function is $m$-supercaloric in $Q_T$.
By lower semicontinuity, there exists an increasing sequence of non-negative smooth functions $\psi_k$, 
such that $\psi_k\rightarrow u$ pointwise in $Q_T$ as $k\to\infty$. 
Note that $\psi_k=0$ in $Q\times(0,t_1)$ for every $k=1,2,\dots$. 

Let $h_k$ be the unique weak solution to the porous medium equation in $(Q\setminus \overline{Q'}) \times (0,T)$ with boundary values
\[
h_k=
\begin{cases}
  \psi_k \quad \text{on}\quad \bd Q' \times [0,T],\\
0\quad  \text{on}\quad \bd Q \times [0,T],\\
0 \quad \text{in}\quad (Q\setminus Q')\times \{0\}.
\end{cases}
\]
Such a solution $h_k$ is continuous up to the boundary so that $h_k\in C((\overline Q\setminus Q')\times [0,T])$. 
By the comparison principle  
\[
h_1\le h_2\le\dots
\quad\text{and}\quad h_k\le u
\quad\text{in}\quad  (Q\setminus \overline{Q'}) \times (0,T)
\] 
for every $k=1,2,\dots$. 
Let $h=\lim_{k\to\infty}h_k$ and
\[
v=
\begin{cases}
  u \quad \text{in}\quad \overline{Q'}\times(0,T),\\
h \quad \text{in}\quad (Q\setminus\overline{Q'})\times(0,T).
\end{cases}
\]
We claim that $v$ is $m$-supercaloric in $Q_T$ and call it the Poisson modification of $u$ in $Q_T$. 
The crucial step in the proof is to show that $v<\infty$ on a dense subset.
To this end, we show that, for any $(x_0,t_0)\in(Q\setminus \overline{Q'})\times(0,T)$, there does not exist a sequence $(x_k,t_k)\rightarrow (x_0,t_0)$, such that 
\[
\lim_{k\rightarrow \infty} h_k(x_k,t_k) = \infty.
\]
Suppose that such a sequence exists for some $(x_0,t_0)$. 
Choose $r>0$ so small that $B(x_0,2r)\subset Q\setminus \overline{Q'}$ and denote 
\[
\theta_k= \frac {C_2 r^2}{h_k(x_k,t_k)^{m-1}},
\quad k=1,2,\dots,
\] 
where $C_2$ is the constant in Lemma \ref{harnack sol}. 
Since $h_k(x_k,t_k) \to \infty$ we have $\theta_k\rightarrow 0$ as $k\to\infty$.
Thus, for $k$ large enough, we have $(t_k-\theta_k, t_k+\theta_k)\subset(0,T)$. 
Let $U$ be the unique positive weak solution to \eqref{eq:elliptic-eq}
with zero boundary values in $B(x_0,r)$. Such a solution is continuous up to the boundary so that $U\in C(\overline B(x_0,r))$.
Then
\[
V_k(x,t)=\frac{U(x)}{(t-t_k+(\lambda-1)\theta_k)^{\frac 1 {m-1}}},
\]
with $\lambda>1$, is a weak solution to the porous medium equation in $B(x_0,r)\times(t_k+\theta_k, T)$. 
We choose $\lambda>1$ so large that
\[
\frac{\|U\|_{L^\infty(B(x_0,r))}}{(\lambda C_2 r^2)^{\frac{1}{m-1}}}\le \frac{1}{C_1},
\]
where $C_1$ and $C_2$ are the constants in Lemma \ref{harnack sol}. 
Then
\begin{align*}
V_k(x,t_k+\theta_k)
&=\frac{U(x)}{(\lambda C_2 r^2)^{\frac 1 {m-1}}}h_k(x_k,t_k) \\
&\le \frac{U(x)}{(\lambda C_2 r^2)^{\frac 1 {m-1}}}C_1 h_k(x,t_k+\theta_k) \\
&\le h_k(x,t_k+\theta_k) 
\quad \text{for every}\quad x\in B(x_k,r),
\end{align*}
where we used Lemma \ref{harnack sol} for the first inequality. 
The comparison principle implies $V_k\le h_k$ in $B(x_k,r)\times(t_k+\theta_k,T)$. 
By letting $k\rightarrow \infty$ we obtain
\[
\frac {U(x)}{(t-t_0)^{\frac{1}{m-1}}} \le h(x,t) 
\quad\text{for every}\quad (x,t)\in B(x_0,r)\times(t_0,T)).
\]
Since $h=\lim_{k\to\infty}h_k\le u$ in $(Q\setminus \overline{Q'}) \times (0,T)$, we have
\[
u(x,t)\ge h(x,t)\ge \frac{U(x)}{(t-t_0)^{\frac{1}{m-1}}}
\quad\text{for every}\quad (x,t)\in B(x_0,r)\times(t_0,T)
\]
and thus $u\notin L^q_{\loc}(\Omega_T)$ whenever $q>m-1$.

Thus we may apply the Harnack type convergence theorem for weak solutions \cite[Lemma 3.4]{kinnunenl08} to conclude that $v$ is  $m$-supercaloric in $Q_T$ as in  \cite[Section 5]{kinnunenl08}.This shows that (i) is valid.

Next we show that (iii) implies (i).
Assume that $\nabla u^m$ exists and 
\[
\nabla u^m \in L_{\loc}^q (\Omega_T)
\quad\text{whenever}\quad 1\le q< 1+ \frac 1 {1+mn}.
\] 
In particular, $\nabla u^m \in L^1_{\loc} (\Omega_T)$. By the definition of a weak derivative this includes $u^m \in L^1_{\loc} (\Omega_T)$.  

Then we show that (i) implies (iv'). 
Let $D\Subset \Omega$ and $\delta\in(0,\frac T2)$. 
Again, we consider the truncations $u_j=\min\{u,j\}$, $j=1,2,\dots$. 
 Since $u_j$ is a weak supersolution, it satisfies the Caccioppoli estimate, see Lemma \ref{caccioppoli}. By assumption, $u\in L_{\loc}^{m-\eps}(\Omega_T)$ for some $\eps\in (0,1)$. 
 We choose a cut-off function $\zeta\in C_0^\infty(\Omega_T)$ such that  $0\le\zeta\le 1$ and $\zeta = 1$ in $D\times(\delta,T-\delta)$.
Lemma \ref{caccioppoli} implies
\begin{align*}
 & \esssup_{t\in(\delta,T-\delta)} \int_D u_j^{1-\eps} \, dx\\
&\le C \left(\iint_{\Omega_T} u_j^{m-\eps} |\nabla \zeta|^2 \, dx \, dt + \iint_{\Omega_T} u_j^{1-\eps} \zeta |\zeta_t| \, dx \, dt\right)\\
&\le C \left(\iint_{\Omega_T} u^{m-\eps} |\nabla \zeta|^2 \, dx \, dt +\iint_{\Omega_T} u^{1-\eps} \zeta |\zeta_t| \, dx \, dt\right)  <\infty.
\end{align*}
Observe that the integrals above are finite, since $u\in L_{\loc}^{m-\eps}(\Omega_T)$ and the support of $\zeta$ is a compact subset of $\Omega_T$.
Since the constant $C$ is independent of $j$, the claim follows by letting $j\rightarrow \infty$. 

Finally we show that (iv') implies (i). 
As above we consider the truncations of $u$, but this time we leave it out in the notation.
Observe that all constants below are independent on the level of truncation.
Let $\zeta$ be a cut-off function as above.
We will show by an iteration argument that $u\in L^q_{\loc}(\Omega_T)$ for some $q>m-1$. 
The idea of the proof is the following. We will show that by (iv') we have $u\in L^{s_0}_{\loc}(\Omega_T)$ for $s_0=\alpha$, and $u\in L^{s_j}_{\loc}(\Omega_T)$ implies $u\in L^{s_{j+1}}_{\loc}(\Omega_T)$ for an increasing sequence of exponents $s_j$. We may iterate this until either $s_j>m-1$ or $s_j=m-1$. In the former case we are done and the latter case is treated separately.

For $\alpha\in(0,1)$ we define
\begin{align*}
s_j=\alpha\left(1+\frac{2j}{n}\right), 
\quad 
r_j=\frac 2 {1+\frac {2j}n}
\quad\text{and}\quad
q_j = 2 \left(1+\frac 2 {n\left(1+\frac{2j}n\right)}\right)
\end{align*}
for $j=0,1,2,\dots$.
We observe that $\frac {s_j}2 q_j = s_{j+1}$ and apply Sobolev's inequality to $w=u^{\frac{s_j}2}$, see Lemma \ref{parabolic sobolev}. This gives
\begin{equation}\label{sobolev iteration1}
\begin{split}
&\iint_{\Omega_T} u^{s_{j+1}} \zeta^{q_j} \, dx \, dt\\
&\le C\iint_{\Omega_T}\left( u^{s_j} |\nabla \zeta|^2 +\zeta^2 u^{s_j-2} |\nabla u|^2\right) \, dx \, dt 
\left( \esssup_{t\in(0,T)} \int_\Omega \zeta^{r_j} u^\alpha \, dx \right)^{\frac2n}.
\end{split}
\end{equation}
By (iv') we have
\[
\esssup_{t\in(0,T)} \int_\Omega \zeta^{r_j} u^\alpha \, dx<\infty.
\]
Lemma \ref{caccioppoli} with $\eps=m-s_j$ implies
\begin{equation}\label{sobolev iteration2}
\begin{split}
 & \iint_{\Omega_T} \zeta^2 u^{s_j-2} |\nabla u|^2 \, dx \, dt\\
&\le C \left( \iint_{\Omega_T} u^{s_j} |\nabla \zeta|^2 \, dx \, dt + \iint_{\Omega_T} u^{s_j-(m-1)} \zeta |\zeta_t| \, dx \, dt \right)< \infty.
\end{split}
\end{equation}
Observe that $u>0$ is a lower semicontinous function and thus it attains its strictly positive minimum $\delta$ on every compact subset of $\Omega_T$.
The same $\delta$ will do for the original $u$ and all truncations.
Thus 
\[u^{s_j-(m-1)} \le \delta^{-(m-1)}u^{s_j}
\] 
for some $\delta>0$ in the support of $\zeta$ and
the second integral on the right-hand side of \eqref{sobolev iteration2} is finite.
Then we consider the first integral on the right-hand side of \eqref{sobolev iteration2}.
We note that in the first step of iteration $s_0=\alpha$ and by (iv') we have
\[
\iint_{\Omega_T} u^{s_0} |\nabla \zeta|^2 \, dx \, dt
\le \esssup_{t\in(0,T)} \int_\Omega \zeta u^\alpha \, dx<\infty.
\]
which implies that $u\in L^{s_0}_{\loc}(\Omega_T)$.
In general, from \eqref{sobolev iteration1}  and \eqref{sobolev iteration2} we may conclude that
if $u\in L^{s_j}_{\loc}(\Omega_T)$ for some $j$, then $u\in L^{s_{j+1}}_{\loc}(\Omega_T)$. 
By iterating this argument, we may step by step increase the local integrability exponent of $u$.
It is essential that we shall use only a finite number of iterations.

This iteration can be done as long as $\eps=m-s_j>0$ and $\eps=m-s_j\ne1$. 
We may assume $\eps>0$ since (i) holds if $s_j>m$. The case $\eps=1$ will be treated separately. Since $s_j$ is an increasing sequence, we can find an index $k$ such that $s_{k-1} <m-1 \le s_k$.
If $s_k>m-1$, then $u\in L^{s_k}_{\loc}(\Omega_T)$ and we are done. 
It remains to consider the case $s_k=m-1$. 
Denote 
\[
r=\frac{2\alpha}{m-1}
\quad\text{and}\quad q= 2\left(1+ \frac {2\alpha}{n(m-1)}\right).
\] 
By applying Sobolev's inequality, see Lemma \ref{parabolic sobolev}, to $w=u^{\frac{m-1}2}$, we obtain
\begin{equation}\label{sobolev iteration3}
\begin{split}
&\iint_{\Omega_T} \zeta^q u^{m-1+\frac {2\alpha}n} \, dx \, dt \\
&\le C \iint_{\Omega_T} \Big( u^{m-1}|\nabla \zeta|^2 + \zeta^2 |\nabla (u^{\frac{m-1}2}) |^2 \Big) \, dx \, dt
 \left (\esssup_{t\in(0,T)} \int_\Omega \zeta^r u^\alpha \, dx \right)^{\frac2n}. 
\end{split}
\end{equation}
Lemma \ref{caccioppoli2} implies
\begin{align*}
&\iint_{\Omega_T}\zeta^2 |\nabla (u^{\frac{m-1}2})|^2 \, dx \, dt\\
&\le C \left(\iint_{\Omega_T} u^{m-1} |\nabla \zeta|^2 \, dx \, dt + \iint_{\Omega_T} |\log(u)| \zeta|\zeta_t| \, dx \, dt \right) <\infty.
\end{align*}
Thus the right-hand side of \eqref{sobolev iteration3} is finite and $u\in L^{m-1+\frac{2\alpha}n}_{\loc}(\Omega_T)$.
\end{proof}

We point out some further implications related to class $\mathfrak{B}$.

\begin{remark}\label{B properties}
A function $u\in\mathfrak{B}$ has the following properties:
\begin{enumerate}
\item  $u\in L_{\loc}^q(\Omega_T)$ for every $q<m+\frac 2n$. 
This is a consequence of a reverse H\"older inequality for supersolutions to the porous medium equation, see \cite{kinnunenl08} and \cite{lehtela16}.
In particular, this implies that  $u\in L_{\loc}^1(\Omega_T)$.
\item There exists a Radon measure $\mu$ on $\R^{n+1}$, such that $u$ is a weak solution to the measure data problem 
\[
u_t-\Delta(u^m) = \mu.
\]
To see this, by the discussion above $u\in L_{\loc}^1(\Omega_T)$ and $\nabla (u^m)\in L_{\loc}^1(\Omega_T)$. 
Thus we may apply the Riesz representation theorem to the non-negative linear operator 
\[
L_u(\varphi)=\iint_{\Omega_T} \left( -u \vp_t + \nabla (u^m) \cdot \nabla \vp\right) \, dx \, dt,
\] 
where $\vp\in C^\infty_0(\Omega_T)$.
\end{enumerate}
\end{remark}

\section{Characterizations for class $\mathfrak{M}$}\label{sec:M}
We say that a positive $m$-supercaloric function $u$ belongs to class $\mathfrak{M}$, if  $u\not \in L_{\loc}^{m-1}(\Omega_T)$.
The friendly giant is a function in class $\mathfrak{M}$.  
The following theorem gives several characterizations for class $\mathfrak{M}$.

\begin{theorem} \label{class M}
Assume that $u$ is a positive $m$-supercaloric function in $\Omega_T$. 
Then the following properties are equivalent:
\begin{itemize}
\item[(i)] $u\not \in L_{\loc}^{q}(\Omega_T)$ for every $ q> m-1$,
\item[(ii)] $u\not \in L_{\loc}^{m-1}(\Omega_T)$,
\item[(iii)] there exists $\delta\in(0,\frac T2)$ such that 
\[
\esssup_{t\in(\delta,T-\delta)} \int_D u(x,t) \, dx = \infty,
\] 
whenever $D\Subset \Omega$ and $|D|>0$. 
\item[(iv)]
there exists $(x_0,t_0)\in \Omega_T$ such that 
\begin{equation}\label{eq:liminf bound}
\liminf_{{\substack{(x,t) \to (x_0,t_0)\\t>t_0}}} u(x,t)(t-t_0)^{\frac 1{m-1}}>0,
\end{equation}
\item[(v)]
there exists $t_0\in (0,T)$ such that 
\[
\lim_{{\substack{(x,t) \to (x_0,t_0)\\t>t_0}}} u(x,t)=\infty 
\quad\text{for every}\quad x_0\in \Omega.
\]
\end{itemize}
\end{theorem}

\begin{remark}
Assume that  (iii) in Theorem \ref{class M} does not hold and let $\alpha\in(0,1)$. 
Then
\[
\esssup_{t\in(\delta,T-\delta)} \int_D u(x,t)^\alpha \, dx
\le\left(\esssup_{t\in(\delta,T-\delta)} \int_D u(x,t) \, dx\right)^\alpha|D|^{1-\alpha}<\infty,
\] 
whenever $D\times(\delta,T-\delta)\Subset \Omega_T$.
This shows that \eqref{eq:ualpha} holds true and thus by Theorem \ref{class B} we conclude that $u\in\mathfrak{B}$.
\end{remark}

The following lemma will be useful for us.

\begin{lemma} \label{blow-up}
Assume that $u>0$ is an $m$-supercaloric function in $\Omega_{T}$ and let $t_0\in(0,T)$. 
Suppose that $B(x_0,8r)\subset\Omega$ and that there exists a sequence $t_j$ belonging to a dense subset of $(t_0,T)$, $j=1,2,\dots$, with $t_j\rightarrow t_0$ as $j\to\infty$, such that
\[
\lim_{j\rightarrow \infty} \int_{B(x_0,r)} u(x,t_j) \, dx =\infty.
\]  
Then there exists $C$, depending only on $n$ and $m$, such that
\[
u(x,t) \ge C\left ( \frac{r^2}{t-t_0}\right)^{\frac{1}{m-1}} 
\quad \text{for every}\quad (x,t)\in B(x_0,4r)\times(t_0,T).
\]
\end{lemma}

\begin{proof}
Denote 
\[
u_\lambda(x,t)=\min \{u(x,t), \lambda \}
\quad\text{with}\quad \lambda>0.
\] 
By \cite[Theorem 3.2]{kinnunenl08} $u_\lambda$ is a weak supersolution in $\Omega_T$ for every $\lambda>0$. 
Let $s>t_0$ to be chosen so that $s-t_0$ is small enough. 
We assume that the times $t_j\in(t_0,T)$, $j=1,2,\dots$ belong to the dense subset of $(0,T)$ where Lemma \ref{Harnack} is applicable. Furthermore, we may assume that
\[
\vint_{B(x_0,r)} u(x,t_j) \, dx > 2\left (\frac{C_1r^2}{s-t_0}\right)^{\frac 1{m-1}}.
\]
Here $C_1$ is the constant  in Lemma \ref{Harnack}. 
Choose $\lambda_j$ such that 
\[
\vint_{B(x_0,r)} u_{\lambda_j}(x,t_j) \, dx = 2\left (\frac{C_1r^2}{s-t_0}\right)^{\frac 1{m-1}}.
\]
Apply Lemma \ref{Harnack} to $u_{\lambda_j}$ at time $t_j$ to obtain
\[
2\left (\frac{C_1r^2}{s-t_0}\right)^{\frac 1{m-1}} 
\le \left ( \frac {C_1 r^2}{s-t_j}\right)^{\frac 1 {m-1}} + C_2 \inf_{Q_j} u_{\lambda_j},
\]
where 
\[Q_j= B(x_0,4r)\times \left(t_j+\frac{\theta_j}2 , t_j+\theta_j\right)
\quad\text{and}\quad \theta_j=\min \left \{s-t_j, \frac{s-t_0}{2^{m-1}} \right\},
\] 
$j=1,2,\dots$. 
By letting $j\rightarrow \infty$, we have
\begin{equation}\label{eq:harnack lower bound}
u(x,t)\ge\frac1{C_2}\left( \frac {C_1 r^2 }{s-t_0}  \right)^{\frac 1 {m-1}}
\ge\frac1{C_2}\left( \frac {C_1 r^2 }{2^m(t-t_0)}  \right)^{\frac 1 {m-1}}
\end{equation}
for every $(x,t)\in B(x_0,4r)\times\left (t_0+\frac{s-t_0}{2^m},t_0+\frac{s-t_0}{2^{m-1}}\right)$.
Finally we observe that for every $t\in (t_0,T)$ we may choose $s>t_0$ such that $t\in\left (t_0+\frac{s-t_0}{2^m},t_0+\frac{s-t_0}{2^{m-1}}\right)$ and thus
\eqref{eq:harnack lower bound} holds for every $ (x,t)\in B(x_0,4r)\times (t_0,T)$.
\end{proof}
 
\begin{proof}[Proof of Theorem \ref{class M}]
We show that (i) $\Longleftrightarrow$ (iii)  and (iii) $\Longrightarrow$ (iv) $\Longrightarrow$ (v) $\Longrightarrow$ (iii).
For the remaining equivalences, see Remark \ref{class B concluded}.

The claim that (i) and (iii) are equivalent follows from Theorem \ref{class B} and Lemma \ref{blow-up}.

We show that (iii) implies (iv). Assume that 
\[
\esssup_{t\in(\delta,T-\delta)} \int_{B(x_0,r)} u(x,t) \, dx = \infty.
\] 
Then we may choose a sequence 
$t_j$, $j=1,2,\dots$ belonging to the dense subset of $(0,T)$ where Lemma \ref{blow-up} is applicable, with $t_j\rightarrow t_0$ as $j\to\infty$, such that
\[
\lim_{j\rightarrow \infty} \int_{B(x_0,r)} u(x,t_j) \, dx = \infty.
\]
By Lemma \ref{blow-up} 
\[
u(x,t) \ge C\left ( \frac {r^2}{t-t_0}\right)^{\frac 1 {m-1}} 
\quad \text{for every}\quad (x,t)\in B(x_0,4r)\times (t_0,T).
\]
This implies \eqref{eq:liminf bound}.

Then we show that (iv) implies (v). Assume that there exists $(x_0,t_0)\in\Omega_T$ such that  \eqref{eq:liminf bound} holds.
Then there exist $r>0$, $\delta>0$ and $\eps>0$ such that 
\[
(t-t_0)^{\frac 1 {m-1}} u(x,t) \ge \eps
\quad\text{for every}\quad (x,t)\in B(x_0,r)\times(t_0,t_0+\delta).
\]
In particular
\[
\vint_{B(x_0,r)} u(x,t)\, dx \ge \eps (t-t_0)^{-\frac1 {m-1}} 
\quad\text{for every}\quad 
t\in (t_0,t_0+\delta).
\]
Lemma \ref{blow-up} shows that 
\[
u(x,t)\ge C\left( \frac {r^2}{t-t_0}\right)^{\frac 1 {m-1}}
\quad \text{for every}\quad 
(x,t)\in B(x_0,4r)\times(t_0,t_0+\delta).
\]  
Thus $B(x_0,4r)\subset \Xi^\perp(t_0)$, where
\[
\Xi^{\perp}(t_0)=\Big\{x_0\in\Omega : \lim_{{\substack{(x,t) \to (x_0,t_0)\\t>t_0}}} u(x,t)=\infty \Big\}.
\]
We may repeat the same argument for any ball intersecting $B(x_0,4r)$. 
Therefore, choosing a suitable chain of balls, we can reach any point in $\Omega$ and conclude that $\Xi^\perp(t_0)=\Omega$. 

Finally we show that (v) implies (iii).
If $\Xi^\perp(t_0)=\Omega$ for some $t_0\in (0,T)$, we have
\[
\int_D u(x,t)\, dx \rightarrow \infty \quad \text{as}\quad t\rightarrow t_0+,
\]
for every set $D\Subset \Omega$ with $|D|>0$. 
Hence there is $\delta>0$ such that 
\[
\esssup_{t\in(\delta,T-\delta)} \int_D u(x,t) \, dx = \infty.
\] 
\end{proof}

\begin{remark}
The friendly giant plays an important role as a minorant for $m$-supercaloric functions which blow up at time $t_0$.
Assume that $u$ is a non-negative $m$-supercaloric function in $\Omega_T$ with the property that
\begin{equation*}
\lim_{\substack{(y,t) \to (x,0)\\t>0}}u(y,t)=\infty
\quad\text{for every}\quad x \in \Omega.
\end{equation*}
Let $\sigma > 0$. The comparison principle gives
\[
u(x,t)\geq U(x) (t+\sigma)^{-\frac{1}{m-1}}\quad \text{for every}\quad (x,t)\in\Omega_{T-\sigma},
\]
where $U$ is a solution to \eqref{eq:elliptic-eq} as in the construction of the friendly giant. 
By letting $\sigma \to 0$, we have 
\[
u(x,t)\geq U(x) t^{-\frac{1}{m-1}}
\quad\text{for every}\quad (x,t)\in\Omega_T.
\]
In particular, 
\[
\liminf_{\substack{(y,t) \to (x,0)\\t>0}}u(y,t)t^{\frac{1}{m-1}}> 0
\quad\text{for every}\quad x\in\Omega.
\]
This shows that an $m$-supercaloric function, with infinite initial values on the whole time slice $\Omega\times\{0\}$, blows up at a rate greater or equal to $t^{\frac 1 {m-1}}$,  see  \cite [p. 111--114]{vazquez07}. 
\end{remark}

The next example shows that an $m$-supercaloric function may blow up faster than the friendly giant.

\begin{example}
Let 
\[
V(x,t)=U(x)e^{\frac 1 {(m-1)t}}, \quad t>0.
\]
Here $U$ is a solution to \eqref{eq:elliptic-eq} as in the construction of the friendly giant. 
We will show that $V$ is a supersolution. A straightforward computation gives
\[
V_t(x,t)-\Delta (V(x,t)^m) 
=e^{\frac 1{(m-1)t}}\left(e^{\frac1t} - \frac 1 {t^2} \right) \frac {U(x)}{m-1} \ge 0.
\]
In a similar manner, we can construct supersolutions that blow up even faster. Let $f : (0,\infty) \rightarrow (0,\infty)$ and define
\[
V(x,t) = U(x)e^{\frac {f(t)} {m-1}}, \quad t>0.
\]
Then $V$ is a supersolution, if $f$ satisfies $f'(t)+e^{f(t)} \ge 0$.
By choosing $f$ in an appropriate way, we see that for any $\varepsilon>0$, we have an $m$-supercaloric function $V$ for which %
$V\notin L^\varepsilon_{\loc}(\Omega_T)$ and $\nabla V\notin L^\varepsilon_{\loc}(\Omega_T)$.
\end{example}

Next we give an explicit example of the dichotomy between classes $\mathfrak{B}$ and $\mathfrak{M}$ by constructing an $m$-supercaloric function as a limit of a sequence of solutions to initial value problems. Depending on the choice of the initial values, the solutions either converge to a Barenblatt type solution, or the limit solution blows up at a rate of the friendly giant. 

\begin{example} 
For $k=1,2,\dots$, consider a weak solution with zero lateral boundary values to the problem
\[
\begin{cases}
\partial_t u_k-\Delta( u_k^m)=0
\quad\text{in}\quad B(0,1)\times(0,\infty),\\
u_k(x,0)=a_k \chi_{ B(0,\frac1k)}
\quad\text{for every}\quad x\in B(0,1).
\end{cases}
\]
Set
\[
v(x,t)=\frac{u_k(x,a_k^{1-m}t)}{a_k},
\]
where $a_k$ are to be chosen later.
The function $v$ satisfies
\[
\begin{cases}
\partial_t v-\Delta(v^m)=0
\quad\text{in}\quad B(0,1)\times(0,\infty),\\
v(x,0)= \chi_{B(0,\frac1k)}(x)
\quad\text{for every}\quad x\in B(0,1).
\end{cases}
\]
Our aim is to compare $v$ to the Barenblatt solution in a suitable space-time cylinder. Let
\[
\begin{split}
\B(x,t)=(t+t_0)^{-\lambda}\left(C-\frac{\lambda(m-1)}{2mn}\frac{\abs{x}^2}{(t+t_0)^{\frac{2\lambda}n}}\right)_+^{\frac1{m-1}},
\end{split}
\]
where $\lambda =\frac{n}{n(m-1)+2}$. We choose
\[
C=C_0\frac{\lambda(m-1)}{2mn}\frac{1}{k^{\frac{2\beta\lambda}n}}
\quad\text{and}\quad
t_0=C_0^{-\frac n{2\lambda}}k^{\beta -\frac n\lambda}.
\]
Here $\beta=(m-1)n$ and $C_0$ is chosen in such a way that $\B(0,0)\le 1$.

For $x\in \bd B\left(0, \frac 1k\right)$, we have 
\[
\B(x,0)=t_0^{-\lambda}\left(C-\frac{\lambda(m-1)}{2mn}\frac{k^{-2}}{t_0^{\frac{2\lambda}n}}\right)_+^{\frac1{m-1}}=0,
\]
since 
\[
C-\frac{\lambda(m-1)}{2mn}\frac{k^{-2}}{t_0^{\frac{2\lambda}n}}= C_0\frac{\lambda(m-1)}{2mn} 
\left(\frac{1}{k^{\frac{2\beta\lambda}n}}-\frac{k^{-2}}{k^{(\beta - \frac n\lambda)\frac{2\lambda}n}}\right) =0.
\]
Then $\B(\cdot,0) \le 1$ in $B\left(0,\frac 1k\right)$ and
$\B(\cdot,0)=0$ in $B(0,1)\setminus B\left(0,\frac 1k\right)$, 
which implies $\B\le v$ in $B(0,1)\times \{0\}$. 
Next, we want to find maximal $\theta>0$ such that $\B$ takes zero lateral boundary values in $B(0,1)\times (0,\theta)$. 
By solving $\theta$ from
 \[
C-\frac{\lambda(m-1)}{2mn}\frac{1}{(\theta+t_0)^{\frac{2\lambda}n}}=0
\] 
we have
\[
\theta = C_0^{-\frac n{2\lambda}}k^\beta (1-k^{\frac n \lambda}).
\]
Then, by the comparison principle, $v\ge \B$ in $B(0,1)\times (0,\theta)$. 
We observe that 
\[
\B(x,\theta) \ge c k^{-\lambda \beta}C^{\frac 1 {m-1}} = c k^{-\lambda \beta \left( 1+ \frac 2 {n(m-1)} \right)} \quad \text{for } x\in B\left(0, \tfrac 14\right).
\]
Here $c=c(C_0,m,n)$. By switching back to the original variables we arrive at
\[
u_k(x,\theta a_k^{1-m})\ge c a_kk^{-\lambda \beta\left(1+\frac2{n(m-1)}\right)} = ca_k k^{\frac{-\beta}{m-1}}.
\] 
Choosing $T=\theta a_k^{1-m}$ gives
\[
u_k(x,T)\ge c T^{-\frac1{m-1}}.
\]
We consider two cases. If 
\[
\frac{k^\beta} {a_k^{m-1}}= \left(\frac{k^n}{a_k}\right)^{m-1} \rightarrow 0
\]
 as $k\rightarrow \infty$, we have 
\[
u(x,T)\ge c T^{-\frac1{m-1}} \quad \text{for every } T>0,
\]
and therefore $u$ is in class $\mathfrak{M}$. 

On the other hand, if
 \[
\frac{a_k}{k^n}\to a<\infty,
\]
 as $k\rightarrow \infty$, then 
\[
\int_{B(0,1)}a_k \chi_{B(0,\frac1k)}(x) \ph(x) \, dx \to  a \ph(0)
\]
for all $\ph\in C_0^\infty (B(0,1))$, showing that $u$ attains the initial value $a \delta$. Thus $u$ is  a Barenblatt type solution, which implies that $u$ is in class $\mathfrak{B}$. 
\end{example}

\section{Infinity sets}\label{sec: infinities}
Assume that $u>0$ is an $m$-supercaloric function in $\Omega_{T}$. 
We consider two sets at time $t_0\in (0,T)$. 
We recall the infinity set that we already encountered in the proof of Theorem \ref{class M} defined as 
\[
\Xi^{\perp}(t_0)=\Big\{x_0\in\Omega : \lim_{{\substack{(x,t) \to (x_0,t_0)\\t>t_0}}} u(x,t)=\infty \Big\}.
\]
In addition, we consider yet another infinity set 
\[
\Xi^{\downarrow}(t_0)=\Big\{x_0\in\Omega: \lim_{{\substack{t \to t_0\\t>t_0}}} u(x_0,t)=\infty\Big\}.
\]
The difference is that in the latter set the limit is taken vertically.
For both sets the times $t\le t_0$ are excluded in the limit procedure.
Clearly $\Xi^{\perp}(t_0) \subset \Xi^{\downarrow}(t_0)$, but the sets are not necessarily same. This can be seen by considering the Barenblatt solution. In this case $\Xi^{\perp}(0)=\emptyset$, but  $\Xi^{\downarrow}(0)=\{0\}$. 
There is an interesting phenomenon: even though the sets may be different, either they both are of $n$-dimensional measure zero, or they occupy the whole time slice. Moreover, the latter alternative cannot occur for $\Omega=\R^n$.

\begin{theorem}\label{infinities}
Assume that $u$ is a positive $m$-supercaloric function in $\Omega_{T}$. 
Then for every $t\in(0,T)$ there are two alternatives: 
\[
\text{either}\quad |\Xi^{\downarrow}(t)| = |\Xi^{\perp}(t)|=0
\quad\text{or}\quad \Xi^{\downarrow}(t)=\Xi^{\perp}(t)=\Omega.
\]
\end{theorem}

\begin{proof}
Let $t_0\in (0,T)$. Since $\Xi^\perp (t_0) \subset \Xi^\downarrow (t_0)$, it suffices to show, that if $|\Xi^\downarrow(t_0)|>0$, then $\Xi^\perp(t_0)=\Omega$. Suppose that $|\Xi^\downarrow(t_0)|>0$. 
Then there exist $x_0\in \Omega$ and $r>0$ such that $B(x_0,8r)\subset \Omega$ and $|\Xi^\downarrow (t_0)\cap B(x_0,r)|>0$. 
Let $k=1,2,\dots$ and $x\in \Xi^\downarrow(t_0)$. 
By definition of the set $\Xi^\downarrow(t_0)$, there exists $t_x^k \in (t_0,T)$ such that $u(x,t) > k$ for every  $t\in (t_0, t_x^k)$.
Let 
\[
E_k=\bigcup\left\{\{x\}\times(0,t_x^k):x\in\Xi^\downarrow (t_0)\right\}
\]
and 
\[
E_k(t)=\{x\in B(x_0,r) : (x,t)\in E_k \},  
\quad t\in (t_0,T).
\]
Observe, that $E_k(t)$ is the projection of $E_k$ to $B(x_0,r)$ and $E_k(t)\subset\Xi^\downarrow (t_0)$.
It is clear that
\[
\Xi^\downarrow (t_0)\cap B(x_0,r) = \bigcap_{k=1}^\infty \bigcup_{j=1}^\infty E_k \left( \frac 1 j \right). 
\]
For a fixed $k$, $E_k(\frac1j)$, $j=1,2,\dots$, is a monotonically increasing sequence of sets. 
Thus 
\[
|\Xi^\downarrow (t_0)\cap B(x_0,r)| 
\le \left | \bigcup_{j=1}^\infty E_k \left( \frac 1 j \right) \right | 
= \lim_{j\rightarrow \infty} \left | E_k \left( \frac 1 j \right) \right|. 
\]
Consequently, there exists an index $j_k$ such that 
\begin{align*}
|E_k(t)|
\ge\left| E_k \left( \frac1{j_k} \right) \right | 
\ge \frac 12 |\Xi^\downarrow (t_0)\cap B(x_0,r)|>0
\quad \text{for every}\quad t\in\left(t_0,\frac1{j_k} \right).
\end{align*}
We may choose a time $t_j\in\left(t_0,\frac1{j_k} \right)$ for every $j=1,2,\dots$ as in Lemma \ref{blow-up} and conclude
\[
\int_{B(x_0,r)} u(x,t_j)\, dx \ge \int_{E_j(t_j)} u(x,t_j)\, dx 
\ge\frac j2 |\Xi^\downarrow (t_0)\cap B(x_0,r)| \rightarrow \infty,
\]
as $j\rightarrow \infty$. 
By Lemma \ref{blow-up} we have $B(x_0,2r)\subset \Xi^\perp(t_0)$. 
Thus the infinity set has expanded.
Finally a chaining  shows that  $\Xi^\perp (t_0)=\Omega$. 
\end{proof}

As a consequence, we obtain characterizations for classes $\mathfrak{B}$ and $\mathfrak{M}$ in terms of the infinity sets. 

\begin{corollary}
Assume that $u$ is a positive $m$-supercaloric function in $\Omega_T$. Then
\begin{itemize}
\item[(i)] $u\in \mathfrak{B}$ if and only if $|\Xi^\downarrow (t)|=0$ for every $t\in(0,T)$ and
\item[(ii)]  $u\in \mathfrak{M}$ if and only if $|\Xi^\downarrow (t)|>0$ for some $t\in(0,T)$.
\end{itemize}
The corresponding claims also hold true for $\Xi^{\perp}(t)$.
\end{corollary}

\begin{proof}
Claim (ii) is a restatement of Theorem \ref{class M} (v) by taking into account Theorem \ref{infinities}. 
Claim (i) follows immediately since classes $\mathfrak{B}$ and $\mathfrak{M}$ are mutually exclusive.
\end{proof}

\begin{remark}\label{class B concluded}
We show the equivalence of  (i) and (ii) in Theorem \ref{class B}. 
It is clear that (i) implies (ii). We show the opposite implication by contradiction.  
Suppose that $u\notin\mathcal{B}$. Then $u\in\mathfrak{M}$. 
By claim (iv) in Theorem \ref{class M}, there exists $(x_0,t_0)\in \Omega_T$, such that 
\[
u(x,t) \ge C (t-t_0)^{-\frac1{m-1}} \quad \text{in a neighbourhood of }(x_0,t_0).
\]
This implies $u\notin L^{m-1}_{loc}(\Omega_T)$. 
By a similar reasoning, we can include the endpoint $\alpha=1$ in \eqref{eq:ualpha}, 
concluding the equivalence of  (iv) and (iv') in Theorem \ref{class B}. By H\"older's inequality (iv) implies (iv'). Again, we shall show the opposite implication by contradiction. If (iv) does not hold, Theorem \ref{class M} implies that $u\in \mathfrak{M}$ and therefore $u\notin \mathfrak{B}$. Thus (iv') does not hold.
\end{remark}

Finally, we show that $m$-supercaloric functions with infinity sets of non-zero measure exist only in the case when the domain $\Omega$ is bounded. 
Here the global bound $u>0$ is decisive.

\begin{theorem}
\label{none}
If  $u:\R^n \times (0,T) \to (0,\infty]$ is $m$-supercaloric, then $u\in\mathfrak{B}$.  
\end{theorem}

\begin{proof}
 We make a counter assumption: there exists a time $t_0\in (0,T)$, such that $|\Xi^\perp(t_0)| >0$. Then, by Theorem \ref{infinities}, we have $\Xi^\perp(t_0)=\R^n$. 
 We consider the friendly giant \eqref{friendly giant} with $\Omega=B(0,1)$ and see that
\[
V(x,t)=U(x)(t-t_0)^{-\frac 1 {m-1}}
\]
 is a weak solution to the porous medium equation in $B(0,1)\times(t_0,T)$. Let $r>0$. 
By the scaling property of solutions, the function 
\[
U\left(\frac x r \right)\left( \frac{r^2}{t-t_0}\right)^{\frac 1 {m-1}}
\]
is a weak solution in $B(0,r)\times (t_0,T)$. By the comparison principle
\[
u(x,t) \ge U\left(\frac x r \right) \left( \frac{r^2}{t-t_0}\right)^{\frac 1 {m-1}}
\quad \text{for every}\quad 
(x,t)\in B(0,r)\times (t_0,T).
\] 
Note that 
\[
\eta = \inf_{x\in B(0,\frac12)}U(x)>0.
\] 
Thus 
\[
u(x,t)\ge  \eta \left( \frac{r^2}{t-t_0}\right)^{\frac 1 {m-1}}
\quad\text{for every}\quad
(x,t)\in B\left(0,\frac r2\right)\times (t_0,T).
\] 
By letting $r\to\infty$, we conclude $u\equiv \infty$. This is a contradiction, as $u$ was assumed to be finite in a dense subset. Hence $|\Xi^\perp (t)| =0 $ for every $t\in(0,T)$ and thus $u\in\mathfrak{B}$. 
\end{proof}

\end{document}